\documentclass[10pt]{amsart}
\usepackage{amsmath}
\usepackage{amssymb}
\usepackage{amsfonts}
\usepackage{amsthm}
\usepackage{enumerate}
\usepackage{hyperref}
\usepackage{color}
\usepackage{amscd}
\usepackage{verbatim}
\usepackage{calc}
\headsep=1cm \numberwithin{equation}{section}
\theoremstyle{plain}
\newtheorem{thm}{Theorem}[section]

\newtheorem{prop}[thm]{Proposition}
\newtheorem{lem}[thm]{Lemma}
\newtheorem{cor}[thm]{Corollary}
\newtheorem{fact}[thm]{Fact}

\theoremstyle{definition}
\newtheorem{defn}[thm]{Definition}

\newtheorem{exmp}[thm]{Example}

\newtheorem{rem}[thm]{Remark}

\newtheorem{defns-rems}[thm]{Definitions and Remarks}
\newtheorem{notas-rems}[thm]{Notations and Remarks}
\newtheorem{exmps-rems}[thm]{Examples and Remarks}

\newcommand{\Hom}{\mbox{Hom}\, }
\newcommand{\Ext}{\mbox{Ext}\, }

\newcommand{\g}{\mbox{G}}

\newcommand{\fm}{\mathfrak{m}}

\begin{document}
%%% ----------------------------------------------------------------------
%%% ----------------------------------------------------------------------
\bibliographystyle{amsplain}
%%% ----------------------------------------------------------------------
%%% ----------------------------------------------------------------------

%%% ----------------------------------------------------------------------
%%% ----------------------------------------------------------------------

\title[Auslander class, $\g_C$ and $C$--projective]{Auslander
class, $\g_C$ and $C$--projective modules \\modulo exact
zero-divisors}

%%% ----------------------------------------------------------------------
%%% ----------------------------------------------------------------------
\bibliographystyle{amsplain}
%%% ----------------------------------------------------------------------
%%% ----------------------------------------------------------------------
     \author[E. Amanzadeh]{Ensiyeh Amanzadeh$^1$}
     \author[M. T. Dibaei]{Mohammad T. Dibaei$^{2}$}

\address{ $^{1,  2}$ Faculty of Mathematical Sciences and Computer,
Kharazmi University,  Tehran,  Iran; and School of Mathematics,
Institute for Research in Fundamental Sciences (IPM),  P.O. Box:
19395-5746,  Tehran,  Iran. }

\email{dibaeimt@ipm.ir}
 \email{en.amanzadeh@gmail.com}

\keywords{semidualizing,  $\g _C$--projective,  $C$--projective,  $C$--flat,  $C$--injective,  Auslander class,  Bass class,  exact zero--divisor .\\
M.T. Dibaei was supported in part by a grant from IPM
(No.$92130110$).}


%%%%%%%%%%%%%%%%%%%%%%%%%%%%%%%%%%%%%%%%%%%%%%%%%%%%%%%%%%%%%%%%%%%%%%%%%%%%%%%%%%%%%%%%%%%%%%%%%%%%%%%%%%%%%%%%%%%%%%%%%%%%%%%%%%%%%%%%%%%%%%%%%%%%%

\begin{abstract}
For a semidualizing module $C$ over a ring $R$, we study the
following classes  modulo exact zero divisors: $\g_C$--projectives,
$\mathcal G_C$; the Auslander
 class $\mathcal A_C$; the Bass class $\mathcal B_C$;
 $\mathcal{P}_C$--projective;
  $ {\mathcal F}_C$--projective; and $ {\mathcal I}_C$--injective dimensions.
\end{abstract}

%%%%%%%%%%%%%%%%%%%%%%%%%%%%%%%%%%%%%%%%%%%%%%%%%%%%%%%%%%%%%%%%%%%%%%%%%%%%%%%%%%%%%%%%%%%%%%%%%%%%%%%%%%%%%%%%%%%%%%%%%%%%%%%%%%%%%%%%%%%%%%%%%%%%%%

\maketitle

%%%%%%%%%%%%%%%%%%%%%%%%%%%%%%%%%%%%%%%%%%%%%%%%%%%%%%%%%%%%%%%%%%%%%%%%%%%%%%%%%%%%%%%%%%%%%%%%%%%%%%%%%%%%%%%%%%%%%%%%%%%%%%%%%%%%%%%%%%%%%%%%%%%%%%%%%%%

\section{Introduction} \label{sec1}

Throughout  $R$ is a commutative noetherian ring with identity element.
In this paper,  we discuss the Auslander class and the Bass class  with respect to a semidualizing $R$--module modulo exact zero--divisors.
Foxby \cite{f} and Vasconcelos \cite{v} independently initiated the study of semidualizing modules.
A finite (i.e. finitely generated) $R$--module $C$ is called {\it semidualizing} if the natural homothety map  $\chi_C^R: R \longrightarrow {\rm Hom}_R(C,  C)$ is an isomorphism and ${\rm Ext}^{>0}_R(C,  C)=0$ {\rm (}see \cite[Definition 1.1]{hj}{\rm )}.  For example every finite projective $R$--module of rank 1 is semidualizing.
A semidualizing $R$--module $C$ gives rise to two full subcategories of the
category of $R$--modules,  namely the Auslander class $\mathcal A_C$ and the Bass class $\mathcal B_C$ defined by Avramov
and Foxby in \cite{f} and \cite{af} (see Definition \ref{h}).
Semidualizing modules and their Auslander and Bass classes have caught attentions of several
authors (see for instance \cite{f}, \cite{c} and \cite{hw}).

According to \cite{hs},  an element $x$ of $R$ is said to be an {\it exact zero--divisor} if $R\neq (0:_R x)\cong R/xR\neq 0$.
Note that in this case if $(0:_R x)=yR$,  then we say that $x,  y$ form a pair of exact zero--divisors on $R$.
This notion has been studied in \cite{hs},  \cite{h},  \cite{ahs},  \cite{bcj} and \cite{dg}.
The concept of exact zero--divisor on a module has been studied in \cite{hs},  \cite{ahs}  and \cite{dg}.

After the introduction,  in Section \ref{sec2},  we bring the necessary tools which are needed for our development.
In Section \ref{sec3},  we study the effect of an exact zero--divisor on a semidualizing module $C$ (see Proposition \ref{1} and Proposition \ref{B}).
In Section \ref{sec4},  we consider the classes $\mathcal A_C$,  $\mathcal B_C$ and $\mathcal G_C$ modulo an exact zero--divisor,
where the class $\mathcal G_C$ denote the class of $\g_C$--{\it projective} modules (see Definition \ref{g}),  and show that if $x$ is an exact zero--divisor on both $R$ and $C$ then $R/xR$ belongs to $\mathcal G_C$ and $\mathcal A_C$ (see Propositions \ref{A} and \ref{C}).
In this section,  we deal with our main goal,  i.e. whether $M/xM$ belongs to the class $\mathcal G_C$,  $\mathcal B_C$,  or
$\mathcal A_C$,  where $x$ is an exact zero--divisor on both $R$ and $R$--module $M$ (see Proposition \ref{J}).
In the final section, \ref{sec5},  we study more closely the class $\mathcal P_C$ of $C$--{\it projective} $R$--modules
which has been studied in \cite{hw} and \cite{tw} before.
We show that if $x$ is an exact zero--divisor on both $R$ and $C$,  then $x$ is an exact zero--divisor on an $R$--module $M$ whenever
$\mathcal P_C$--{\rm pd}$(M)<\infty$,  which implies that
$\mathcal P_{C/xC}$--{\rm pd}$(M/xM)\leq \mathcal P_C$--{\rm pd}$(M)$ (see Propositions \ref{F} and \ref{G}).

%%%%%%%%%%%%%%%%%%%%%%%%%%%%%%%%%%%%%%%%%%%%%%%%%%%%%%%%%%%%%%%%%%%%%%%%%%%%%%%%%%%%%%%%%%%%%%%%%%%%%%%%%%%%%%%%%%%%%%%%%%%%%%%%%%%%%%%%%%%%%%%%%%%%%%%%%%%%%%%%%%
\section{Preliminaries} \label{sec2}

Throughout the paper,  $C$ is a semidualizing $R$--module. First we
recall the definition of exact zero--divisor from \cite[Definition]{hs}
and \cite[Definition 2.3]{dg}.
\begin{defn} \label{k}
Let $M$ be an $R$--module. An element $x$  of $R$ is called an {\it exact zero}--{\it divisor} on $M$ if  $xM\neq 0$,  $xM\neq M$
and there is $y\in R$ such that the sequence of multiplication maps $ M\stackrel{x}{\longrightarrow } M \stackrel{y}{\longrightarrow } M\stackrel{x}{\longrightarrow } M$ is exact. In this case we say that
$x,  y$ form a {\it pair} of exact zero--divisors on $M$.
\end{defn}

Here we recall some facts about a pair of exact zero--divisors.

\begin{fact} \cite[ Proposition 2.4]{dg} \label{a}
If $x$ is an exact zero--divisor on an $R$--module $M$,  then $0:_M x\cong M/xM$.
\end{fact}

The following fact was proved for $n=0$ in \cite[Lemma 2.10]{dg},  but it is clearly true for any $n\geq0$.
\begin{fact} \label{b}
Assume that $x, y$ form a pair of exact zero--divisors on $R$. Let $M$ be an $R$--module,  $n$ a non-negative integer.
Consider the following statements.
 \begin{itemize}
\item[(i)] $x,  y$ form a pair of exact zero--divisors on $M$.
\item[(ii)]  ${\rm Ext}^i_R(R/xR, M)=0$ for all $i>n$.
\item[(iii)] ${\rm Tor}^R_i(R/xR, M)=0$ for all $i>n$.
\end{itemize}
Then {\rm (i)}$\Rightarrow${\rm (ii)}$\Leftrightarrow${\rm (iii)}.
If one of the following conditions holds true,  then the statements {\rm  (i),  (ii)} and {\rm (iii)} are equivalent.

{\rm  (a)} $xM\neq 0$ and $xM\neq M$.

{\rm  (b)} $R$ is local and $M$ is finite.
\end{fact}

\begin{fact} \cite[Proposition 2.13]{dg} \label{c}
Assume that  $x,  y$ form a pair of exact zero--divisors on both $R$ and an $R$--module $M$ and that $N$ is an $R/xR$--module. Then the following statements hold true for all $i\geq0$.
 \begin{itemize}
\item[(i)] ${\rm Ext}^i_R(N, M)\cong {\rm Ext}^i_{R/xR}(N, M/xM)$.
\item[(ii)] ${\rm Ext}^i_R(M, N)\cong {\rm Ext}^i_{R/xR}(M/xM, N)$.
\item[(iii)] ${\rm Tor}^R_i(M, N)\cong {\rm Tor}^{R/xR}_i(M/xM, N)$.
\end{itemize}
\end{fact}
%%%%%%%%%%%%%%%%%%%%%%%%%%%%%%%%%%%%%%%%%%%%

The class of $\g_C$--projective $R$--modules,  $\mathcal G _C$,  the Auslander class $\mathcal A _C$ and the Bass class $\mathcal B _C$
have been investigated modulo a regular element ( see \cite[Propositions 5.7, 5.8, 5.9 and 5.10]{c}).
In this paper,  we are concerned with these classes modulo an exact zero--divisor.

\begin{defn}\cite[Definition 2.7]{hj}\label{g}
The class $\mathcal G _C(R)$ (or simply $\mathcal G _C$) consists of $\g_C$--{\it projective} $R$--modules,
i.e. the class of all finite $R$--modules $M$ which satisfy the following conditions.
\begin{itemize}
\item[(i)] the natural homomorphism $\delta^C_M: M \longrightarrow {M^{{\dag}{ \dag}}}$ is an isomorphism,  where ${(-)}^{\dag}={\rm Hom}_R(-,  C)$.
\item[(ii)] ${\rm Ext}^{>0}_R(M,  C)=0={\rm Ext}^{>0}_R(M^{\dag},  C)$.
\end{itemize}

\end{defn}

\begin{defn} \label{h}
The $Auslander\ class$ $\mathcal A _C(R)$ (or $\mathcal A _C$) with respect to $C$
is the class of all $R$--modules $M$ satisfying the following conditions.

(1) The natural map $ \gamma^C_M: M \longrightarrow {\rm Hom}_R(C,  C\otimes_R M)$ is an isomorphism.

(2) ${\rm Tor}^R_{>0}(C,  M)=0={\rm Ext}^{>0}_R(C,  C\otimes_R M)$.

Dually,  the $Bass\ class$ $\mathcal B _C(R)$ (or $\mathcal B _C$) with respect to $C$ is the class of all $R$--modules $M$ satisfying the following conditions.

(1) The evaluation map $ \xi ^C_M: C\otimes_R {\rm Hom}_R(C,  M)\longrightarrow M$ is an isomorphism.

(2) $ {\rm Ext}^{>0}_R(C,  M)=0={\rm Tor}^R_{>0}(C,  {\rm Hom}_R(C,  M))$.
\end{defn}

%%%%%%%%%%%%%%%%%%%%%%%%%%%%%%%%%
We need the following classes which are studied in \cite{hw},  \cite{tw} and \cite{swty}.

\begin{defn} \cite[Definition 5.1]{hw}\label{i}
The classes of $C$--{\it injective},  $C$--{\it projective} and $C$--{\it flat} modules are defined,  respectively,  as
\begin{itemize}
\item[] $ {\mathcal I}_C=\{\ {\rm Hom}_R(C,  I)\ |\ I\ {\rm is\ an\ injective}\ R$--module$\}, $
\item[] $ {\mathcal P}_C=\{\ C\otimes_R P\ |\ P\ {\rm is\ a\ projective}\ R$--module$\}, $
\item[] $ {\mathcal F}_C=\{\ C\otimes_R F\ |\ F\ {\rm is\ a\ flat}\ R$--module$\}.$
\end{itemize}
They are the classes of injective,  projective and flat $R$--modules,  respectively,  when $C=R$.
\end{defn}

\begin{rem}\cite[1.4 and Proposition 5.3]{hw}\label{e}
For any $R$--module $M$ there exists an {\it augmented proper} $ {\mathcal P}_C$--{\it projective resolution},
 that is,  a complex
$$X^+ = \cdots\stackrel{{\partial}_2^{X}}{\longrightarrow }C\otimes_R P_1\stackrel{{\partial}_1^{X}}{\longrightarrow }C\otimes_R P_0 \stackrel{{\partial}_0^{X}}{\longrightarrow }M\longrightarrow 0$$
such that ${\rm Hom}_R(C\otimes_R Q,  X^+)$ is exact for all projective $R$--module $Q$. The truncated complex
$$X = \cdots\stackrel{{\partial}_2^{X}}{\longrightarrow }C\otimes_R P_1\stackrel{{\partial}_1^{X}}{\longrightarrow }C\otimes_R P_0\longrightarrow 0$$
is called a {\it proper} $ {\mathcal P}_C$--{\it projective resolution} of $M$. Note that $X^+$ need not be exact unless $C=R$.

An {\it augmented proper} $ {\mathcal F}_C$--{\it projective resolution} for $M$ is defined similarly.

Dually,  for any $R$--module $N$ there exists an {\it augmented proper} $ {\mathcal I}_C$--{\it injective resolution},  that is,  a complex
$$Y^+ = 0\longrightarrow N \longrightarrow {\rm Hom}_R(C,  I^0) \stackrel{{\partial}^0_Y}{\longrightarrow } {\rm Hom}_R(C,  I^1) \stackrel{{\partial}^1_Y}{\longrightarrow }\cdots\ $$
such that ${\rm Hom}_R(Y^+,  {\rm Hom}_R(C,  I))$ is exact for all injective $R$--module $I$.
\end{rem}

\begin{defn} \cite[1.6 and Proposition 5.3]{hw}\label{f}
The $ {\mathcal P}_C$--{\it projective dimension} of an $R$--module $M$ is
$${\mathcal P}_C-{\rm pd}(M)=\inf \big\{\sup X \mid X\ {\rm is\ a\ proper}\ {\mathcal P}_C-{\rm projective\ resolution\ of}\ M \big\},$$ where $\sup X= \sup\{n\mid  X_n\neq 0\}$.
The modules of zero ${\mathcal P}_C$--projective dimensions are the non-zero modules in $ {\mathcal P}_C$; and we set ${\mathcal P}_C$--{\rm pd}$(0)=-\infty$.
The $ {\mathcal F}_C$--{\it projective dimension},  denoted $ {\mathcal F}_C$--pd(--),  is defined similarly and
the $ {\mathcal I}_C$--{\it injective dimension},  denoted $ {\mathcal I}_C$--id(--),  is defined dually.
\end{defn}

%%%%%%%%%%%%%%%%%%%%%%%%%%%%%%%%%%%%%%%%%%%%%%%%%%%%%%%%%%%%%%%%%%%%%%%%%%%%%%%%%%%%%%%%%%%%%%%%%%%%%%%%%%%%%%%%%%%%%%%%%%%%%%%%%%%%%%%%%%%%%%%%%%%%%%%%%%%%%%%%%%

\section{Semidualizing modules via exact zero--divisors} \label{sec3}

Note that if $x\in R$ is non--zero,  then $xC\neq 0$. By Nakayama's lemma,  $xC=C$ if and only if $(x)=R$.
Now if  $x,  y$  form a pair of  exact zero--divisors on $R$ and ${\rm pd}(C)<\infty$,  then Auslander-Buchsbaum formula implies that $C$ is projective and so $x,  y$  form a pair of  exact zero--divisors on $C$ by Definition \ref{k}. Moreover,  if $R$ is Cohen-Macaulay local ring with dualizing module $\omega$,  that is a semidualizing module with finite injective dimension,  then $x,  y$  form also a pair of  exact zero--divisors on $\omega$,  by Fact \ref{b} (see also \cite[Corollary 2.12]{dg}).
In general,  the authors do not know  whether a pair of exact zero--divisors on $R$ is also a pair of exact zero--divisors on $C$.
In the following proposition one can see that it holds true under certain conditions.

\begin{prop}\label{1}
Let $(R,  {\frak m},  k)$ be a Cohen-Macaulay local ring that is not Gorenstein,  with dualizing module $\omega$. Let $f:R\longrightarrow S$ be a flat local ring homomorphism such that $S/{\frak m}S$ is not Gorenstein. Assume that $x,  y\in S$ form a pair of exact zero--divisors on $S$ such that
${\rm fd}_R(S/xS)<\infty$. Then $S\otimes_R \omega$ is a semidualizing $S$--module which is not a dualizing $S$--module and
${\rm pd}_S(S\otimes_R \omega)=\infty$. Moreover,  $x,  y$ form a pair of exact zero--divisors on $S\otimes_R \omega$.
\end{prop}
\begin{proof}
Set $B=S\otimes_R \omega$.
By \cite[Theorem 5.6]{c},  $B$ is a semidualizing $S$--module while $B$ is not dualizing module.
Since $\omega \ncong R$ and ${\rm Ann}_R(\omega)=0$,  one observes that $\omega$ is not cyclic. Then $B$ is not cyclic and so $B\ncong S$. Now,  as $S$ is local, Auslander-Buchsbaum formula implies that ${\rm pd}_S(B)=\infty$.

As $S$ is flat $R$--module,  we obtain the isomorphisms
$${\rm Tor}^S_i(S/xS,  B)\cong{\rm Tor}^R_i(S/xS\otimes_S S,  \omega)\cong {\rm Tor}^R_i(S/xS,  \omega)$$
for all $i$,  where the first isomorphism is from \cite[Corollary 10.61]{r}.
Thus ${\rm Tor}^S_{i>n}(S/xS,  B)=0$ for some non-negative integer $n$. Now,  by Fact \ref{b},  $x,  y$ form a pair of exact zero--divisors on $B$.
\end{proof}

We take the following example from \cite[Example 2.3.1]{w} to justify Proposition \ref{1}.
\begin{exmp}
Let $R=k[X,  Y]/(X, Y)^2$,  whenever $k$ is a field. Then $R$ is a local artinian ring that is not Gorenstein.
By \cite[Theorem 6.1]{c},  $\omega={\rm Hom}_k(R,  k)$ is dualizing $R$--module,  as $R$ is free $k$--module of rank 3.
Set $S=R[U,  V,  W,  Z]/(U^2,  VW,  VZ)$. Then $S$ is free $R$--module and $S/{\fm S}\cong k[U,  V,  W,  Z]/(U^2,  VW,  VZ)$ is not Cohen-Macaulay,  where
${\frak m}$ is the maximal ideal of $R$. If $u$ is the image of $U$ in $S$,  then $u,  u$ form a pair of exact zero--divisors on $S$.
We have an $R$--isomorphism $S/uS\cong R[V,  W,  Z]/(VW,  VZ)$ and so $S/uS$ is free $R$--module.
Thus $u,  u$ form also a pair of exact zero--divisors on the semidualizing $S$--module $S\otimes_R \omega$,  by Proposition \ref{1}.
Note that $S\otimes_R \omega$ is not a dualizing $S$--module with ${\rm pd}_S(S\otimes_R \omega)=\infty$.
\end{exmp}

It is easy to see that if $x,  y$ form a pair of  exact zero--divisors on $R$,  then $R/xR\in {\mathcal G}_R$.
The following proposition shows that $R/xR \in {\mathcal G}_{ C}$,  whenever $x,  y$ form also a pair of  exact zero--divisors on $C$.
\begin{prop}\label{A}
If $x,  y$ form a pair of exact zero--divisors on both $R$ and $C$,  then
${R / xR}  \in {\mathcal G}_{ C}$.
\end{prop}
\begin{proof}
Since $x$ is an exact zero--divisor on $C$,  we conclude that
$ {\rm Hom}_{R} ({R / xR},  C)\cong {C / xC}$,  by Fact \ref{a}. So we have
$$\begin{array}{llll}
{R / xR}  \cong   {\rm Hom} _{R} ({R / xR},  R) \hspace{-0.25cm}  &\cong {\rm Hom} _{R} ({R / xR},  {\rm Hom}_R(C,  C)) \\
                                                 &\cong {\rm Hom}_R ( {R / xR}\otimes_R C,  C) \\
                                                 &\cong {\rm Hom}_R ({\rm Hom} _{R} ({R / xR},  C),  C).  \\
\end{array}$$
As $x,  y$ form a pair of exact zero--divisors on $C$,  one has $ {\rm Ext }^{>0}_R(R/xR,  C)=0$,  by Fact \ref{b}.
As $C\otimes_R P \in {\mathcal G}_{ C}$,  for any projective $R$-module $P$,  ${\rm Ext} ^{>0}_R(C\otimes_R P,  C)=0$.
By \cite[Theorem 10.62]{r} there is a spectral sequence
$${\rm E}_2^{p, q}= {\rm Ext }^{p}_R({\rm Tor}^R_q(C,  R/xR),  C) \underset{p}{\Rightarrow} {\rm Ext }^{n}_R(R/xR,  {\rm Hom}_R(C,  C)).$$
Since ${\rm Tor}^R_q(C,  R/xR)=0$ for all $q>0$,  by Fact \ref{b},   the spectral sequence  collapses on the $p$-axis
and so $ {\rm Ext }^{n}_R(C\otimes_R R/xR,  C)\cong  {\rm Ext }^{n}_R(R/xR,  {\rm Hom}_R(C,  C))$ for all $n\geq0$. Then, for all $n>0$,
$$\begin{array}{lll}
{\rm Ext }^{n}_R({\rm Hom} _{R} ({R / xR},  C),  C) & \cong {\rm Ext }^{n}_R(C\otimes_R R/xR,  C) \hspace{-0.25cm}\\& \cong {\rm Ext }^{n}_R(R/xR,  {\rm Hom}_R(C,  C))\\
                                                                                               & \cong {\rm Ext }^{n}_R(R/xR,  R)=0\\
\end{array}$$
and the result follows.
\end{proof}

In the following result we use Proposition \ref{A},  to achieve semidualizing $R/xR$--modules.
\begin{prop}\label{B}
Let $B$ be a finite $R$--module. Assume that $x,  y$ form a pair of exact zero--divisors on both $R$ and $B$.
Then the following statements are equivalent.
\begin{itemize}
\item[(i)] $B$ is a semidualizing $R$--module.
\item[(ii)] $B/xB$ and $B/yB$ are semidualizing $R/xR$-- and $R/yR$--modules,  respectively.
\end{itemize}
\end{prop}

\begin{proof}
(i)$\Rightarrow$(ii). It is enough to show that $B/xB$ is a semidualizing $R/xR$--module.
As $x$ is an exact zero--divisor on $B$ we have
$$\begin{array}{llll}
{R / xR} \hspace{-0.25cm}& \cong {\Hom}_R ({\rm Hom} _{R} (R/xR,  B),  B)\\
                         & \cong {\Hom}_{R} (B/xB,  B)\\
                         & \cong {\Hom}_{R/xR} (B/xB,  B/xB),  \\
\end{array}$$
where the first,  second and third isomorphisms are obtained by Proposition \ref{A},  Fact \ref{a} and Fact \ref{c},  respectively. For all $i>0$
$$\begin{array}{lll}{\rm Ext }^{i}_{R/xR} (B/xB,  B/xB) & \cong {\rm Ext }^{i}_{R} (B/xB,  B)\\
                                         &  \cong {\rm Ext }^{i}_{R} ({\rm Hom} _{R} (R/xR,  B),  B)\\ & =0,
                                         \end{array}$$
 where the first isomorphism is obtained by Fact \ref{c} and the equality holds by Proposition \ref{A}.

(ii)$\Rightarrow$(i).
Let $z\in \{x,  y\}$. By Fact \ref{c},  we have $\Ext^{i}_{R} (B/zB, B) \cong {\Ext }^{i}_{R/zR} (B/zB, B/zB)$ for all $ i \geq 0$. Thus
$\Ext^{i}_{R} (B/zB,  B)=0$ for all $i>0$ and  $\Hom_{R} (B/zB,  B)\cong {R/zR}$,  by the fact that $B/zB$ is a semidualizing $R/zR$-module.
Since $x,  y$ form a pair of exact zero--divisors on $B$,  there is an exact sequence
\begin{equation}\label{e1}{0 \longrightarrow  B/yB \longrightarrow B \longrightarrow B/xB \longrightarrow 0.}
\end{equation} Now applying $\Hom_R(-,  B)$ to (\ref{e1}) gives $\Ext^{>0}_R(B,  B)=0$ and the following commutative diagram
$$\begin{array}{lllllllll}
0 & \longrightarrow & \Hom_{R}(B/xB,  B) & \longrightarrow & \Hom_{R} (B,  B) & \longrightarrow & \Hom_{R} (B/yB,  B) &\longrightarrow & 0 \\
  &                 &\begin{array}{ll} \ \ \ {\Bigg\uparrow} \cong \end{array} &               &\begin{array}{ll} \ \ \ {\Bigg\uparrow} \chi^{R}_{B} \end{array}  &         &\begin{array}{ll}\ \ \ {\Bigg\uparrow}\cong\end{array} &                &  \\
0 & \longrightarrow & \ \ \    R/xR         & \longrightarrow &   \ \ \ \   R       & \longrightarrow & \ \ \  R/yR            &\longrightarrow & 0. \\
\end{array}$$
Thus $ \chi^{R}_{B}  $ is an isomorphism and so $B$ is a semidualizing $R$--module.
\end{proof}

\begin{cor}
Let $(R,  {\frak m},  k)$ be a Cohen-Macaulay local ring,  $D$ a finite $R$--module. Assume that
$x,  y$ form a pair of exact zero--divisors on both $R$ and $D$.
If $D/xD$ is a  dualizing $R/xR$--module and  $D/yD$ is a semidualizing $R/yR$--module,  then
$D$ is a  dualizing $R$--module.
\end{cor}
\begin{proof}
By Proposition \ref{B},  $D$ is a semidualizing $R$--module. By \cite[corollary 2.14]{dg},  we have
 ${\rm id}_{R}(D)={\rm id}_{R/xR}(D/xD)<\infty$ and so $D$ is a dualizing $R$--module.
\end{proof}

Note that the converse of the corollary was proved in \cite[Corollary 2.12]{dg}.

%%%%%%%%%%%%%%%%%%%%%%%%%%%%%%%%%%%%%%%%%%%%%%%%%%%%%%%%%%%%%%%%%%%%%%%%%%%%%%%%%%%%%%%%%%%%%%%%%%%%%%%%%%%%%%%%%%%%%%%%%%%%%%%%%%%%%%%%%%%%%%%%%%%%%%%%
\section{The classes $\mathcal{A}_C$,  $\mathcal{B}_C$ and $\mathcal{G}_C$} \label{sec4}

This section contains our main results. Our aim is to find when $M/xM$ belongs to ${\mathcal G}_{ C}$,  $\mathcal{B}_C$,  or $\mathcal{A}_C$,
where $M$ is an $R$--module and $x$ is an exact zero--divisor on both $R$ and $M$.
\begin{prop}\label{C}
If $x, y$ form a pair of exact zero--divisors on both $R$ and $C$,  then $R/xR \in \mathcal{A}_C$.
\end{prop}
\begin{proof}
By Proposition \ref{B} and Fact \ref{c},  there are isomorphisms
$$R/xR \cong {\rm Hom}_{R/xR} ({C / xC},  {C / xC}) \cong  {\rm Hom}_{R} (C,  C\otimes_R R/xR).$$
As $x,  y$ form a pair of exact zero--divisors on both $R$ and $C$,  we get $ {\rm Tor}^R_i(C,  R/xR)=0$ for all $i>0$,  by Fact \ref{b}.
Also,  for all $i>0$,  ${\rm Ext}^{i}_{R} (C,  C\otimes_R R/xR) \cong {\rm Ext}^{i}_{R/xR} ({C / xC},  {C / xC})=0$,
where the isomorphism and equality are obtained by Fact \ref{c} and Proposition \ref{B},  respectively.
\end{proof}

Now,  one may pay attention to the results \cite[Proposition 5.3 and Theorem 6.5]{c} via exact zero--divisors.
\begin{cor}\label{K}
 Assume that $x,  y$ form a pair of exact zero--divisors on both $R$ and $C$. If $T$ is an $R/xR$--module,  then the following statements hold true.
\begin{itemize}
\item[(i)] If $T$ is finite, then $T \in \mathcal G_C(R)$ if and only if $T \in \mathcal G_{C/xC}(R/xR)$.
\item[(ii)]    $T \in \mathcal A_C(R)$ if and only if $T \in \mathcal A_{C/xC}(R/xR)$.
\item[(iii)]    $T \in \mathcal B_C(R)$ if and only if $T \in \mathcal B_{C/xC}(R/xR)$.
\end{itemize}
\end{cor}
\begin{proof}
(i) holds by \cite[Theorem 6.5]{c} and Proposition \ref{A}.
$ {\rm (ii)}\ $ and $ {\rm (iii)}\ $ follow from \cite[Proposition 5.3]{c} and Proposition \ref{C}.
\end{proof}

The following example shows that, for an $R$--module $M$ and an
exact zero--divisor $x$ on both $R$ and $C$, the assumption $M\in
\mathcal G_C(R)$ does not necessarily imply that $M/xM \in \mathcal
G_{C/xC}(R/xR)$.
\begin{exmp}\cite[Example 5.4.14]{w}\label{exa1}
Let $k$ be a field,  and set $R=k[[X,  Y]]/(XY)$. Now $x,  y$ form a pair of exact zero--divisors on $R$.
Thus $M=R/yR$ belongs to $ \mathcal G_R(R)$ by Proposition \ref{A}.
As $M/xM\cong k$,  $R/xR \cong k[[Y]]$ and $ \Ext^1_{k[[Y]]}(k, k[[Y]])\cong k$ is non--zero, $M/xM\not\in \mathcal G_{R/xR}(R/xR)$.
\end{exmp}

We do not know whether, for an $R$--module $M$, $M\in \mathcal
G_C(R)$ implies that $M/xM \in \mathcal G_{C/xC}(R/xR)$,  whenever
$x$ is an exact zero--divisor on all $R$,  $C$ and $M$. But there is
a partial converse in the following result.
\begin{prop}\label{D}
Assume that $M$ is an $R$--module and that $x,  y$ form a pair of exact zero--divisors on $R$,  $C$ and $M$.
\begin{itemize}
\item[(i)] If $M/xM \in \mathcal A_{C/xC}(R/xR)$ and $M/yM \in \mathcal A_{C/yC}(R/yR)$,  then $M \in \mathcal A_C(R)$.
\item[(ii)] If $M/xM \in \mathcal B_{C/xC}(R/xR)$ and $M/yM \in \mathcal B_{C/yC}(R/yR)$,  then $M \in \mathcal B_C(R)$.
\item[(iii)] If $M$ is finite,  $M/xM \in \mathcal G_{C/xC}(R/xR)$ and $M/yM \in \mathcal G_{C/yC}(R/yR)$,  then $M \in \mathcal G_C(R)$.
\end{itemize}
\end{prop}
\begin{proof}
(i). As $x, y$ form a pair of exact zero--divisors on $M$,  there is an exact sequence
\begin{equation}\label{e2}{0 \longrightarrow  M/yM \longrightarrow M \longrightarrow M/xM \longrightarrow 0}.
\end{equation}
Let $z\in \{x, y\}$. As $x, y$ form a pair of exact zero-divisors on $C$,  $${\rm Tor}^R_i(C, M/zM)\cong {\rm Tor}^{R/zR}_i(C/zC,  M/zM)=0$$ for all $i>0$,  by Fact \ref{c}.
By applying $C\otimes_R-$ to (\ref{e2}),  we get ${\rm Tor}^R_{>0}(C, M)=0$. In particular,  one has the exact sequence
\begin{equation}\label{e3}{0\longrightarrow  C\otimes_R M/yM \longrightarrow C\otimes_R M \longrightarrow C\otimes_R M/xM \longrightarrow 0}.
\end{equation}
Also
$$\begin{array}{llll}
{\rm Ext}^{i}_{R}(C,  C\otimes_R M/zM) \hspace{-0.25cm}& \cong {\rm Ext}^{i}_{R/zR}(C/zC,  C\otimes_R M/zM)\\
                                      \hspace{-0.25cm}& \cong {\rm Ext}^{i}_{R/zR}(C/zC,  C/zC\otimes_{R/zR} M/zM)=0
\end{array}$$
for all $i>0$,  by Fact \ref{c}.
So by applying ${\rm Hom}_{R} (C,  -)$ to (\ref{e3}),  we get ${\rm Ext}^{>0}_{R}(C,  C\otimes_R M)=0$ and the following commutative diagram
{\small $$\begin{array}{lllllllll}
0 \hspace{-0.25cm}& \longrightarrow \Hom_{R}(C,  C\otimes_R M/yM) \hspace{-0.25cm}& \longrightarrow \Hom_{R}(C,  C\otimes_R M) \hspace{-0.25cm}& \longrightarrow  \Hom_{R} (C,  C\otimes_R M/xM) \hspace{-0.25cm}&\longrightarrow  0 \\
  &   \begin{array}{ll} \ \ \ \ \ \ \ \ \ \ \ \ \ \ \ \ {\Bigg\uparrow} \cong \end{array} &  \begin{array}{ll} \ \ \ \ \ \ \ \ \ \ \ \ \ \ \ {\Bigg\uparrow} {\gamma ^C_M}\end{array} & \begin{array}{ll}\ \ \ \ \ \ \ \ \ \ \ \ \ \ \ \ \ \ {\Bigg\uparrow}\cong\end{array} &  \\
0 \hspace{-0.25cm}& \longrightarrow \ \ \ \ \ \ \ \ \ M/yM & \longrightarrow\ \ \ \ \ \ \ \ \ \  M &\longrightarrow\ \ \ \ \ \ \ \ \ \ \ M/xM &\longrightarrow  0 \\
\end{array}$$ }
Now the result follows.

The proofs of (ii) and (iii) are similar.
\end{proof}

We end this section with our main result indicating when $M/xM$ belongs to the classes $\mathcal{G}_C$, $\mathcal{B}_C$, or $\mathcal{A}_C$.
\begin{prop}\label{J}
Assume that $M$ is an $R$--module and that $x,  y$ form a pair of exact zero--divisors on both $R$ and $M$.
\begin{itemize}
\item[(i)] If $M \in \mathcal G_C$,  then $M/xM \in \mathcal G_C$ if and only if $x,  y$ form also a pair of exact zero--divisors on
${\rm Hom}_R(M,  C)$.
\item[(ii)] If $M \in \mathcal B_C$,  then $M/xM \in \mathcal B_C$ if and only if $x,  y$ form also a pair of exact zero--divisors on
${\rm Hom}_R(C,  M)$.
\item[(iii)] If  $M \in \mathcal A_C$ is finite ,  then $M/xM \in \mathcal A_C$ if and only if $x,  y$ form also a pair of exact zero--divisors on $C\otimes_R M$.
\end{itemize}
\end{prop}

\begin{proof}
(i). As $M\in \mathcal G_C$,  we have $M\otimes_R P\in \mathcal G_C$ for every finite projective $R$--module $P$ and so
${\rm Ext}^{>0}_R(M\otimes_R P,  C)=0$. Thus,  by \cite[Theorem 10.62]{r},  there is a spectral sequence
$${\rm E}_2^{p, q}={\rm Ext}_R^p({\rm Tor}^R_q(M,  R/xR),  C)\underset{p}{\Rightarrow} {\rm Ext}_R^n(R/xR, {\rm Hom}_R(M,  C)).$$
By Fact \ref{b}, ${\rm Tor}^R_{>0}(M, R/xR)=0$. Hence
\begin{equation}\label{e4}{{\rm Ext}_R^n(M/xM,  C)\cong {\rm Ext}_R^n(R/xR,  {\rm Hom}_R(M,  C))\ {\rm for\ all}\ n\geq 0}.
\end{equation}
As $M\in \mathcal G_C$, it follows that ${\rm Hom}_R(M, C)\overset{x}{\longrightarrow}{\rm Hom}_R(M, C)$ is neither zero nor isomorphism.
Assume that $M/xM \in \mathcal G_C$ then,  by (\ref{e4}) and Fact \ref{b},  $x, y$ form a pair of exact zero--divisors on ${\rm Hom}_R(M, C)$.
For the converse,  ${\rm Ext}_R^{>0}(M/xM, C)=0$,  by Fact \ref{b} and (\ref{e4}).
In order to obtain the other conditions, we consider the isomorphisms
$$\begin{array}{llll}
M/xM \cong {\rm Hom}_R(R/xR, M) \hspace{-0.25cm}&  \cong  {\rm Hom}_R(R/xR, {\rm Hom}_R({\rm Hom}_R(M, C), C)) \\
                                \hspace{-0.25cm} &  \cong  {\rm Hom}_R(R/xR\otimes_R {\rm Hom}_R(M, C), C) \\
                                 \hspace{-0.25cm} &  \cong  {\rm Hom}_R({\rm Hom}_R(R/xR, {\rm Hom}_R(M, C)), C) \\
                                 \hspace{-0.25cm} &  \cong  {\rm Hom}_R({\rm Hom}_R(M/xM, C), C), \\
\end{array}$$
where the first and forth isomorphisms are from Fact \ref{a} and assumptions; the second isomorphism follows from $M \in \mathcal{G}_C$; and the third and the last ones are Hom-tensor adjointness.
Replacing $M$ with ${\rm Hom}_R(M, C)$ in (\ref{e4}), implies
${\rm Ext}_R^n({\rm Hom}_R(M,  C)\otimes_R R/xR,  C)\cong {\rm Ext}_R^n(R/xR,  {\rm Hom}_R({\rm Hom}_R(M,  C),  C)\cong {\rm Ext}_R^n(R/xR,  M)=0$
for all $n>0$.
As $x, y$ form a pair of exact zero--divisors on ${\rm Hom}_R(M,  C)$,  one has
$${\rm Hom}_R(M,  C)\otimes_R R/xR\cong {\rm Hom}_R(R/xR,  {\rm Hom}_R(M,  C))\cong {\rm Hom}_R(M/xM,  C).$$
Therefore ${\rm Ext}_R^{>0}({\rm Hom}_R(M/xM, C), C)=0$.

(ii). Let $P_{\bullet}: \cdots \longrightarrow P_1\longrightarrow P_0\longrightarrow C\longrightarrow 0$ be a projective resolution of $C$.
Consider the third quadrant bicomplex $W={\rm Hom}_R(P_{\bullet}, {\rm Hom}_R(F_{\bullet}, M))$, where
$F_{\bullet}: \cdots \longrightarrow R \stackrel{y}{\longrightarrow } R\stackrel{x}{\longrightarrow }R\longrightarrow  R/xR\longrightarrow  0$
is the free resolution of $R/xR$.
Let $ {\rm ^{I} E}$ and ${\rm ^{II} E}$ denote the spectral sequences determined by the first filtration and second filtration of ${\rm Tot}(W)$,  respectively.
Thus ${\rm ^{I} E}_2^{p,  q}={\rm Ext}^p_R(C,  {\rm Ext}^q_R(R/xR,  M))\underset{p}{\Rightarrow} H^n({\rm Tot}(W))$.
As,  by Fact \ref{b}, ${\rm Ext}^{>0}_R(R/xR,  M)=0$ one has ${\rm ^{I} E}_2^{n,  0}\cong {\rm Ext}^n_R(C,  M/xM)\cong H^n({\rm Tot}(W))$ for all $n\geq 0$.
On the other hand there is an isomorphism of bicomplexes ${\rm Hom}_R(P_{\bullet},  {\rm Hom}_R(F_{\bullet},  M))\cong {\rm Hom}_R(F_{\bullet},  {\rm Hom}_R(P_{\bullet},  M))$. Thus ${\rm ^{II} E}_2^{p,  q}\cong {\rm Ext}^p_R(R/xR,  {\rm Ext}^q_R(C,  M))\underset{p}{\Rightarrow} H^n({\rm Tot}(W))$.
As $M \in \mathcal{B}_C$,  we have ${\rm Ext}^{>0}_R(C,  M)=0$ and so
${\rm ^{II} E}_2^{n,  0}\cong {\rm Ext}^n_R(R/xR,  {\rm Hom}_R(C,  M))\cong H^n({\rm Tot}(W))$ for all $n\geq 0$.
Therefore
\begin{equation}\label{e5}{{\rm Ext}^n_R(C,  M/xM)\cong{\rm Ext}^n_R(R/xR,  {\rm Hom}_R(C,  M))\ {\rm for\ all}\ n\geq 0}.
\end{equation}
As $M \in \mathcal{B}_C$, $C\otimes_R {\rm Hom}_R(C, M)\cong M$, therefore ${\rm Hom}_R(C, M)\overset{x}{\longrightarrow}{\rm Hom}_R(C, M)$ is neither zero nor surjective.
Now assume that $M/xM \in \mathcal{B}_C$ then $x, y$ form a pair of exact zero--divisors on ${\rm Hom}_R(C, M)$,  by (\ref{e5}) and Fact \ref{b}.
For the converse,  ${\rm Ext}^{>0}_R(C,  M/xM)=0$,  by Fact \ref{b} and (\ref{e5}).

For the remaining conditions, we first consider the following isomorphisms
$$\begin{array}{llll}
M/xM \cong R/xR\otimes_R M \hspace{-0.25cm}&  \cong  R/xR\otimes_R C\otimes_R {\rm Hom}_R(C,  M)\\
                           \hspace{-0.25cm}&  \cong  C\otimes_R {\rm Hom}_R(R/xR,  {\rm Hom}_R(C,  M)) \\
                           \hspace{-0.25cm}&  \cong  C\otimes_R {\rm Hom}_R(C,  {\rm Hom}_R(R/xR,  M)) \\
                           \hspace{-0.25cm}&  \cong  C\otimes_R {\rm Hom}_R(C,  M/xM),  \\
\end{array}$$
where the second isomorphism is from $M \in \mathcal{B}_C$;  the third and the last isomorphisms are obtained by Fact \ref{a}.
As $M \in \mathcal{B}_C$,  one has ${\rm Hom}_R(C, M)\in \mathcal{A}_C$ (see\cite[Theorem 2.8]{tw}) and so
${\rm Hom}_R(C, M)\otimes_R P\in \mathcal A_C$ which gives ${\rm Tor}_{>0}^R(C,  {\rm Hom}_R(C,  M)\otimes_R P)=0$
for every projective $R$--module $P$.
By \cite[Theorem 10.59]{r},  there is a first quadrant spectral sequence
$${\rm E}_{p, q}^2={\rm Tor}_p^R(C, {\rm Tor}_q^R({\rm Hom}_R(C, M), R/xR))\underset{p}{\Rightarrow}{\rm Tor}_n^R(C\otimes_R {\rm Hom}_R(C, M), R/xR).$$
As $x$ is an exact zero--divisor on ${\rm Hom}_R(C, M)$, ${\rm Tor}_{>0}^R({\rm Hom}_R(C,  M),  R/xR)=0$,  by Fact \ref{b}. Therefore, for all $n\geq 0$, one has
\begin{equation}\label{e6}{{\rm Tor}_n^R(C,  {\rm Hom}_R(C,  M)\otimes_R R/xR)\cong {\rm Tor}_n^R(C\otimes_R {\rm Hom}_R(C,  M),  R/xR)}.
\end{equation}
As $x,  y$ form a pair of exact zero--divisors on both $M$ and ${\rm Hom}_R(C,  M)$,  we have
$$\begin{array}{llll}
{\rm Hom}_R(C,  M)\otimes_R R/xR \hspace{-0.25cm}&\cong {\rm Hom}_R(R/xR, {\rm Hom}_R(C, M))\\
                                 \hspace{-0.25cm}&\cong {\rm Hom}_R(C, {\rm Hom}_R(R/xR, M))\\
                                  \hspace{-0.25cm}&\cong {\rm Hom}_R(C, M/xM). \\
\end{array}$$
Hence, for all $n>0$, we have
$$\begin{array}{llll}
{\rm Tor}_n^R(C,  {\rm Hom}_R(C,  M/xM)) \hspace{-0.25cm}&\cong {\rm Tor}_n^R(C\otimes_R {\rm Hom}_R(C,  M),  R/xR)\\
                                         \hspace{-0.25cm}&\cong {\rm Tor}_n^R(M,  R/xR)=0,\\
\end{array} $$
which fulfils the requirement.

(iii). As $M \in \mathcal{A}_C$ is finite,  the map $C\otimes_R M \overset{x}{\longrightarrow}C\otimes_R M$ is neither zero nor isomorphism.
Also,  for every projective $R$--module $P$,  $M \otimes_R P\in \mathcal{A}_C$ and so ${\rm Tor}_{>0}^R(C,  M \otimes_R P)=0$.
By \cite[Theorem 10.59]{r},  there is a first quadrant spectral sequence
$${\rm E}_{p,  q}^2={\rm Tor}_p^R(C,  {\rm Tor}_q^R(M,  R/xR))\underset{p}{\Rightarrow} {\rm Tor}_n^R(C\otimes_R M,  R/xR).$$
As, by  Fact \ref{b},  ${\rm Tor}_{>0}^R(M,  R/xR)=0$, ${\rm Tor}_n^R(C,  M/xM)\cong {\rm Tor}_n^R(C\otimes_R M,  R/xR)$ for all $n\geq 0$.

Assume that $M/xM \in \mathcal{A}_C$. Now it follows that $x,  y$ form a pair of exact zero--divisors on $C\otimes_R M$,  by Fact \ref{b}.
For the converse,  from $M \in \mathcal{A}_C$ we have $C\otimes_R M \in \mathcal{B}_C$,  by \cite[Theorem 2.8]{tw}.
Since $x,  y$ form a pair of exact zero--divisors on both $C\otimes_R M$ and ${\rm Hom}_R(C,  C\otimes_R M)\cong M$,  one has
$C\otimes_R M/xM \cong (C\otimes_R M)/{x(C\otimes_R M)}\in \mathcal{B}_C$,  by (ii). Thus $M/xM\in \mathcal{A}_C$,  by \cite[Theorem 2.8]{tw}.
\end{proof}

%%%%%%%%%%%%%%%%%%%%%%%%%%%%%%%%%%%%%%%%%%%%%%%%%%%%%%%%%%%%%%%%%%%%%%%%%%%%%%%%%%%%%%%%%%%%%%%%%%%%%%%%%%%%%%%%%%%%%%%%%%%%%%%%%%%%%%%%%%%%%%%%%%%%%%%%

\section{The classes ${\mathcal P}_C$,  ${\mathcal F}_C$  and ${\mathcal I}_C$} \label{sec5}

We observed in Fact \ref{b} and \cite[Proposition 2.18]{dg}(a) that, if $x, y$ form a pair of exact zero--divisors on $R$ and if $M$ is an $R$--module such that $M\overset{x}{\longrightarrow}M$ is neither zero nor epimorphism, then $x, y$ form also a pair of exact zero--divisors on $M$ whenever one of the conditions ${\rm id}(M)<\infty$,  ${\rm pd}(M)<\infty$,  or ${\rm fd}(M)<\infty$ holds true.
In this section, we provide a positive answer to the question whether the same property is true if one replaces the above homological dimension by
${\mathcal I}_C$--{\rm id}$(M)$, ${\mathcal P}_C$--{\rm pd}$(M)$, or ${\mathcal F}_C$--{\rm pd}$(M)$, respectively.

\begin{prop} \label{E}
Assume that $x, y$ form a pair of exact zero--divisors on both $R$ and $C$. Let $M$ be an $R$--module such that $M\overset{x}{\longrightarrow}M$ is neither zero nor epimorphism.
Then $M$ admits $x, y$ as a pair of exact zero-divisors if it is either in
${\mathcal I}_C(R)$,  ${\mathcal P}_C(R)$,  or ${\mathcal F}_C(R)$. In any such case $M/xM$ belongs to ${\mathcal I}_{C/xC}(R/xR)$,  ${\mathcal P}_{C/xC}(R/xR)$,  or ${\mathcal F}_{C/xC}(R/xR)$, respectively.
\end{prop}
\begin{proof}
We prove the case $M\in {\mathcal I}_C(R)$. In this case, $M=\Hom_R(C, I)$ for some injective $R$--module $I$. By \cite[Corollary 10.63]{r},  we have
$ {\rm Ext}_R^i(R/xR,  {\rm Hom}_R(C,  I))\cong {\rm Hom}_R({\rm Tor}_i^R(R/xR,  C),  I)$
for all $i\geq 0$. Now by Fact \ref{b},  ${\rm Tor}_{>0}^R(R/xR,  C)=0$ and thus ${\rm Ext}_R^{>0}(R/xR, M)=0$.
Therefore $x, y$ form a pair of exact zero-divisors on $M$.
Now we have
$$\begin{array}{llll}
R/xR\otimes_R {\rm Hom}_R(C, I) \hspace{-0.25cm}& \cong {\rm Hom}_R({\rm Hom}_R(R/xR, C), I)\\
 \hspace{-0.25cm}& \cong {\rm Hom}_R(C\otimes_R R/xR, I)\\
 \hspace{-0.25cm}& \cong {\rm Hom}_R(C, {\rm Hom}_R(R/xR,  I))\\
 \hspace{-0.25cm}& \cong {\rm Hom}_{R/xR}(C/xC, {\rm Hom}_R(R/xR, I)),
\end{array}$$
where the first and the third isomorphisms follow from the Hom evaluation homomorphism and adjointness,  respectively.
The second and the last ones hold,  by Facts \ref{a} and \ref{c},  respectively.
Note that ${\rm Hom}_R(R/xR,  I)$ is an injective $R/xR$--module and that $C/xC$ is a semidualizing $R/xR$--module (Proposition \ref{B}).
Hence $M/xM \in  \mathcal {I}_{C/xC}(R/xR)$.
\end{proof}

The following result is consistent with \cite[Proposition 2.18]{dg}(a).

\begin{prop}\label{F}
Assume that $x,  y$ form a pair of exact zero--divisors on both $R$ and $C$.
Let $M$ be an $R$--module such that $xM\neq 0$ and $xM\neq M$.
If ${\mathcal P}_C$--{\rm pd}$(M)$,  ${\mathcal F}_C$--{\rm pd}$(M)$,  or ${\mathcal I}_C$--{\rm id}$(M)$ is finite,
then $x,  y$ form a pair of exact zero--divisors on $M$.
\end{prop}
The proof follows by the following lemma and Fact \ref{b}.

\begin{lem}\label{H}
Assume that $M$ is an $R$--module and that $x, y$ form a pair of exact zero--divisors on both $R$ and $C$.
Let $n$ be a non-negative integer. Then the following statements hold true.
\begin{itemize}
\item[(i)]  If ${\mathcal P}_C$--{\rm pd}$(M)\leq n$,  then ${\rm Tor}^R_{>n}(R/xR,  M)=0$.
\item[(ii)] If ${\mathcal F}_C$--{\rm pd}$(M)\leq n$,  then ${\rm Tor}^R_{>n}(R/xR,  M)=0$.
\item[(iii)] If ${\mathcal I}_C$--{\rm id}$(M)\leq n$,  then ${\rm Ext}_R^{>n}(R/xR,  M)=0$.
\end{itemize}
\end{lem}
\begin{proof}
(i). Let $M\neq 0$. We prove by induction on $n$. If $n=0$ then $M=C\otimes_R P$ for some projective $R$--module $P$
and thus ${\rm Tor}^R_i(R/xR,  C\otimes_R P)\cong {\rm Tor}^R_i(R/xR,  C)\otimes_R P$ for all $i\geq 0$.
As, by Fact \ref{b}, ${\rm Tor}^R_{>0}(R/xR,  C)=0$, we have ${\rm Tor}^R_{>0}(R/xR,  M)=0$.

Let $n>0$. By \cite[Corollary 2.10]{tw},  there exists an exact sequence
$$0\longrightarrow N \longrightarrow C\otimes_R P\longrightarrow M \longrightarrow 0, $$
 where $P$ is projective $R$--module and ${\mathcal P}_C$--{\rm pd}$(N)\leq n-1$.
Now, the long exact sequence
$$ {\rm Tor}^R_i(R/xR,  C\otimes_R P)\longrightarrow {\rm Tor}^R_i(R/xR,  M)\longrightarrow{\rm Tor}^R_{i-1}(R/xR,  N)$$
and our induction hypothesis imply ${\rm Tor}^R_{>n}(R/xR, M)=0$.

The proof of (ii) is a modification of the proof of (i) for which we use \cite[Proposition 5.2]{swty}.
The proof of (iii) is similar to (i).
\end{proof}

Our final result is consistent with \cite[Proposition 2.18]{dg}.
\begin{prop}\label{G}
Assume that $x, y$ form a pair of exact zero--divisors on both $R$ and $C$.
Let $M$ be an $R$--module such that $xM\neq 0$ and $xM\neq M$. Set $\overline{(-)} = (-)\otimes_R R/xR$.
The following statements hold true.
\begin{itemize}
\item[(i)]  If ${\mathcal P}_C$--{\rm pd}$(M)<\infty$,  then ${\mathcal P}_{\overline{C}}$--{\rm pd}$(\overline{M})\leq {\mathcal P}_C$--{\rm pd}$(M)$.
\item[(ii)] If ${\mathcal F}_C$--{\rm pd}$(M)<\infty$,  then ${\mathcal F}_{\overline{C}}$--{\rm pd}$(\overline{M})\leq {\mathcal F}_C$--{\rm pd}$(M)$.
\item[(iii)] If $M$ is finite with ${\mathcal I}_C$--{\rm id}$(M)<\infty$,  then ${\mathcal I}_{\overline{C}}$--{\rm id}$(\overline{M})\leq {\mathcal I}_C$--{\rm id}$(M)$.
\item[(iv)] If $R$ is local and $M$ is finite,  then equality holds in {\rm(i),  (ii)} and {\rm (iii)}.
\end{itemize}
\end{prop}
\begin{proof}
(i). As ${\mathcal P}_C$--{\rm pd}$(M)$ is finite,  one has $M\in {\mathcal B}_C$,  by \cite[Corollary 2.9]{tw}. By Proposition \ref{F}, $x, y$ form a pair of exact zero--divisors on $M$ and thus the map ${\rm Hom}_R(C, M)\overset{x}{\longrightarrow}{\rm Hom}_R(C, M)$ is neither zero nor surjective.
On the other hand, by \cite[Theorem 2.11]{tw},  ${\rm pd}_R({\rm Hom}_R(C, M))=\mathcal P_C$--{\rm pd}$(M)$ is finite.
Therefore $x, y$ form a pair of exact zero--divisors on ${\rm Hom}_R(C, M)$ and
${\rm pd}_{\overline{R}}(\overline{{\rm Hom}_R(C, M)})\leq {\rm pd}_R({\rm Hom}_R(C, M))$,  by \cite[Proposition 2.18]{dg}.
Now we have
$$\begin{array}{llll}
\overline{\Hom_R(C, M)}\hspace{-0.25cm}&\cong \Hom_R({\overline{R}},  \Hom_R(C, M))\\
\hspace{-0.25cm}&\cong \Hom_R(\overline{C}, M)\\
\hspace{-0.25cm}&\cong\Hom_{\overline{R}}(\overline{C}, \overline{M}),\\
\end{array}$$
where the first and second isomorphisms hold by Fact \ref{a} and adjointness. The third isomorphism follows from Fact \ref{c} and Proposition \ref{F}.
Therefore \cite[Theorem 2.11]{tw} will result
$$\begin{array}{llll}
\mathcal P_{\overline{C}}\hspace{-0.05cm}-\hspace{-0.1cm}{\rm pd}(\overline{M})\hspace{-0.25cm}&={\rm pd}_{\overline{R}}(\Hom_{\overline{R}}(\overline{C}, \overline{M}))\\
 \hspace{-0.25cm}& = {\rm pd}_{\overline{R}}(\overline{{\rm Hom}_R(C, M)})\\
 \hspace{-0.25cm}& \leq {\rm pd}_R({\rm Hom}_R(C, M))\\
 \hspace{-0.25cm}&={\mathcal P}_C\hspace{-0.1cm}-\hspace{-0.1cm}{\rm pd}(M).
\end{array}$$
(ii). As ${\mathcal F}_C$--{\rm pd}$(M)<\infty$,  one has $M\in {\mathcal B}_C$,  by \cite[Lemma 5.1]{swty}, and so the map ${\rm Hom}_R(C, M)\overset{x}{\longrightarrow}{\rm Hom}_R(C, M)$ is neither zero nor surjective.
By \cite[Proposition 5.2]{swty},  ${\rm fd}_R({\rm Hom}_R(C, M))={\mathcal F}_C$--{\rm pd}$(M)$ is finite.
Thus $x, y$ form a pair of exact zero--divisors on ${\rm Hom}_R(C, M)$ and
${\rm fd}_{\overline{R}}(\overline{{\rm Hom}_R(C, M)})\leq {\rm fd}_R({\rm Hom}_R(C, M))$,  by \cite[Proposition 2.18]{dg}.
Therefore
$$\begin{array}{llll}
{\mathcal F}_{\overline{C}}\hspace{-0.05cm}-\hspace{-0.1cm}{\rm pd}(\overline{M})\hspace{-0.25cm}&={\rm fd}_{\overline{R}}(\Hom_{\overline{R}}(\overline{C}, \overline{M}))\\
 \hspace{-0.25cm}& ={\rm fd}_{\overline{R}}(\overline{{\rm Hom}_R(C, M)})\\
 \hspace{-0.25cm}& \leq{\rm fd}_R({\rm Hom}_R(C, M))\\
 \hspace{-0.25cm}&={\mathcal F}_C\hspace{-0.1cm}-\hspace{-0.1cm}{\rm pd}(M).
\end{array}$$
(iii). As ${\mathcal I}_C$--{\rm id}$(M)<\infty$,  one has $M\in
{\mathcal A}_C$,  by \cite[Corollary 2.9]{tw}. By Proposition
\ref{F}, $x, y$ form a pair of exact zero--divisors on $M$. It
follows that the map $C\otimes_R M
\overset{x}{\longrightarrow}C\otimes_R M$ is neither zero nor
isomorphism. By \cite[Theorem 2.11]{tw},  ${\rm id}_R(C\otimes_R
M)={\mathcal I}_C$--{\rm id}$(M)$ is finite. Then  $x, y$ form a
pair of exact zero--divisors on $C\otimes_R M$ and ${\rm
id}_{\overline{R}}(\overline{C\otimes_R M})\leq {\rm
id}_R(C\otimes_R M)$,  by \cite[Proposition 2.18]{dg}. By
\cite[Theorem 2.11]{tw}, one has $ {\mathcal I}_{\overline{C}}$
--${\rm id}(\overline{M})= {\rm
id}_{\overline{R}}(\overline{C}\otimes_{\overline{R}}
\overline{M})={\rm id}_{\overline{R}}(\overline{C\otimes_R M})$.
Thus ${\mathcal I}_{\overline{C}}$ --${\rm id}(\overline{M})
 \leq{\rm id}_R(C\otimes_R M)={\mathcal I}_C$ --${\rm id}(M).
$

(iv). The equalities follow by \cite[Proposition 2.18]{dg} and the
proofs of (i), (ii) and (iii), respectively.
\end{proof}

%%%%%%%%%%%%%%%%%%%%%%%%%%%%%%%%%%%%%%%%%%%%%%%%%%%%%%%%%%%%%%%%%%%%%%%%%

{\bf Acknowledgement.} The authors thanks Sean Sather-Wagsaff for reading the manuscript and his communications. They are also grateful to the referee for her/his invaluable comments.

%%%%%%%%%%%%%%%%%%%%%%%%%%%%%%%%%%%%%%%%%%%%%%%%%%%%%%%%%%%%%%%%%%%%%%%%%

%%%%%%%%%%%%%%%%%%%%%%%%%%%%%%%%%%%%%%%%%%%%%%%%%%%%%%%%%%%%%%%%%%%%%%%%%%


\begin{thebibliography}{10}

\bibitem{af} L. L. Avramov and H. B. Foxby,  {\it Ring homorphisms and finite Gorenstein dimension},  Proc. London Math. Soc. (3) 75 (1997),
241--270.

\bibitem{ahs} L. L. Avramov, I. B. Henriques and L. M. $\c{S}$ega,  {\it Quasi-Complete Intersection Homomorphisms},
arXiv:1010.2143v2 [math.AC] 21 February 2011.

\bibitem{bcj} P. A. Bergh,  O. Celikbas and D. A. Jorgensen,  {\it Homological Algebra Modulo Exact zero–divisors},
arXiv:1012.3010v1 [math.AC] 13 December 2010.

\bibitem{c} L. W. Christensen,  {\it Semi-dualizing complexes and their Auslander categories},  Trans. Amer. Math. Soc. 353 (2001), 1839--1883.


\bibitem{dg} M. T. Dibaei and M. Gheibi,  {\it Sequence of exact zero--divizors},  arXiv:1112.2353v3 [math.AC] 23 April 2012.


\bibitem{f} H. B. Foxby,  {\it Gorenstein modules and related modules},  Math. Scand. 31 (1972),  267--284.


\bibitem{hs} I. B. Henriques and L. M. $\c{S}$ega,  {\it Free resolutions over short Gorenstein local rings},  Math. Z. 267 (2011), 645--663.


\bibitem{h} H. Holm,  {\it Construction of totally reflexive modules from an exact pair of zero divisors},  Bull. London
Math. Soc. 43 (2010),  278--288.

\bibitem{hj}  H. Holm and P. J{\o}rgensen, {\it Semi-dualizing modules and related Gorenstein
homological dimensions}, Journal of Pure and Applied Algebra 205 (2006), 423--445.

\bibitem{hw}  H. Holm and D. White,  {\it Foxby equvalence over associative rings},  J. Math. Kyoto Univ. 47 (2007),  781--808.


\bibitem{r} J. J. Rotman,  {\it An introduction to homological algebra},  Second Edition, Springer Universitext, 2009.


\bibitem{swty} M. Salimi, S. Sather-Wagstaff, E. Tavasoli and S. Yassemi,  {\it Relative Tor functors with respect to a semidualizing module}, Algebras and Representation Theory, DOI 10.1007/s10468-012-9389-4.


\bibitem{w} S. Sather-Wagstaff,  {\it Semidualizing modules},  http://www.ndsu.edu/pubweb/\~ssatherw/DOCS/sdm.pdf


\bibitem{tw} R. Takahashi and D. White,  {\it Homological aspects of semidualizing modules},  Math. Scand. 106 (2010), 5--22.

\bibitem{v} W. V. Vasconcelos,  {\it Divisor theory in module categories},  North-Holland Math. Stud., vol. 14,  North-Holland Publishing
Co.,  Amsterdam,  1974.


\end{thebibliography}
\end{document}